\documentclass{article}
\usepackage{amsmath,amsthm,amsopn,amstext,amscd,amsfonts,amssymb,verbatim}
\usepackage{color}
\usepackage{natbib}

\def\tr{\mathop{\rm tr}\nolimits}

\def \build#1#2#3{\mathrel{\mathop{#1}\limits^{#2}_{#3}}}

\renewenvironment{abstract}
                 {\vspace{6pt}
                  \begin{center}
                  \begin{minipage}{5in}
                  \centerline{\textbf{Abstract}}
                  \noindent\ignorespaces
                 }
                 {\end{minipage}\end{center}}
\newtheorem{theorem}{\textbf{Theorem}}[section]

\theoremstyle{definition}

\newtheorem{remark}{\textbf{Remark}}[section]

\title{\Large \textbf{Multimatricvariate distribution under elliptical models}}
\author{
  \textbf{Jos\'e A. D\'{\i}az-Garc\'{\i}a} \thanks{Corresponding author\newline
   {\bf Key words.}  Bimatrix variate, matricvariate, matrix variate, random matrices, matrix variate elliptical distributions.\newline
    2000 Mathematical Subject Classification. 62E15; 60E05}\\
  {\normalsize Universidad Aut\'onoma de Chihuahua} \\
  {\normalsize Facultad de Zootecnia y Ecolog\'{\i}a} \\
  {\normalsize Perif\'erico Francisco R. Almada Km 1, Zootecnia} \\
  {\normalsize 33820 Chihuahua, Chihuahua, M\'exico}\\
  {\normalsize E-mail: jadiaz@uach.mx}\\
  \textbf{Francisco J. Caro-Lopera}\\
  {\normalsize Departament of Basic Sciences} \\
  {\normalsize Universidad de Medell\'{\i}n} \\
  {\normalsize Medell\'{\i}n, Colombia} \\
  {\normalsize E-mail: fjcaro@udem.edu.co} \\[2ex]
}
\date{}
\begin{document}
\maketitle

\begin{abstract}
A new family of matrix variate distributions indexed by elliptical models are proposed in this work. The
so called \emph{multimatricvariate distributions} emerge as a generalization of the bimatrix variate
distributions based on matrix variate gamma distributions and independence. Some properties and special
cases of the multimatricvariate distributions are also derived.  Two new interesting Jacobians in the
area are also provided. Finally, an application for time dependent data of DNA molecules is studied.
\end{abstract}

\section{Introduction}\label{sec:1}
New challenges for modeling random variables in the real world, demand advanced statistical analysis of
the experiments and the phenomena. In that sense, the statistical literature proposes a number of
approaches dealing with complex situations. For example, the assumption of normality has ruled the
statistics for decades, but the existence of multiple situations which do not fulfill such requirement
has forced new insights based on robust distributions. For example, the elliptical contoured
distributions provide a wide class of flexible functions with different kurtosis, tails and properties.
This family includes most of the well known univariate, multivariate or matrix variate symmetric
distributions, such as: normal, Pearson VII (t and Cauchy)and II, Laplace, Bessel, Kotz, among many
others.

The elliptical distributions constitute a robust approach against normality but also demand advanced
theory for dealing with probabilistic dependent models. The literature reports that including the
dependency is not an easy task, and  the theorists prefers models under a weak initial condition based on
independence. Then the experts, users of the statistics, do not have any choice for modeling under
dependence.

For example, some experiments involving two random variables, say $X$ and $Y$, usually assume independent
probabilistic densities; in this case, $dF_{X}(x)$ and $dG_{Y}(y)$, respectively. Then, the joint density
function of $X$ and $Y$, is just given by the product  $dF_{X}(x)dG_{Y}(y)$.

However, the complex phenomena generally do not accept such independence assumption.  For example,
consider the temperature and rainfall in a context of certain model in hydrology. For practical uses the
variables are considered independent ruled by a normal and a gamma distribution, respectively. But any
expert knows that the occurrence of one depends on the occurrence of the other variable.

Then we ask for a joint distribution when $X$ and $Y$ are not probabilistically independent.  Moreover,
we demand that the marginal distributions correspond to the distributions under independence.

In the univariate case, several bivariate type distributions has been proposed, the list includes the
pairs: gamma-beta, gamma-gamma, beta-beta among others. They have been applied in hydrology, finance and
other areas, see for example \citet{ln:82},\citet{cn:84}, \citet{ol:03}, \citet{n:07, n:13}  and
\citet{spj:14} and the references therein. In a general setting, the bimatrix variate type distribution
have been studied by several authors, see \citet{or:64}, \citet{dggj:10a, dggj:10b,dggj:11}, and
\citet{brea:11}, and the references therein . In particular, \citet{e:11} presents a complete exposition
of the advances in this difficult topic. As usual the univariate and multivariate cases are the most
common in literature, but the matrix variate case counts few works.

Moreover, there is an important problem in the referred literature of the univariate, multivariate and
matrix variate cases. The associated density functions, are based on the joint density function of three
independent random variables (scalar, vector or matrix). Then under certain change of variables the
"dependent" model is obtained. Clearly, the underlying initial independence keeps the model far from an
authentic probabilistic dependent model.

Alternatively, the present article, provides a natural extension to $n$ random matrices in the context of
a real subjacent probabilistic dependence.  In this case we will start with the joint density function of
$k$ dependent probabilistically random matrices. Then,  we propose the joint density function of two or
more random matrices under dependence, where the corresponding marginals distribution are known. This
will be termed \textit{multimatricvariate distributions}. In Section \ref{sec:2}  new results on special
Jacobians and some integrals are derived. The central results are presented in Section \ref{sec:3}. Some
properties of multimatricvariate distributions are obtained in Section \ref{sec:4}. Finally an
application of DNA is showed in Section \ref{sec:5}.

\section{Preliminary results}\label{sec:2}

Some basic results about computations of Jacobians and density functions are provided next.

\begin{theorem}\label{teo1}
Let $\mathbf{X} \in \Re_{n \times m}$ and $\mathbf{Y} \in \Re_{n \times m}$, matrices with mathematically
independent elements.
\begin{description}
  \item[i)] Define $\mathbf{Y} = \mathbf{X} (\mathbf{I}_{m}-\mathbf{X}'\mathbf{X})^{-1/2}$. Then
     \begin{equation}\label{eq2}
        (d\mathbf{Y}) = \left|\mathbf{I}_{m} - \mathbf{X}'\mathbf{X}\right|^{-(n+m+1)/2} (d\mathbf{X}).
     \end{equation}
  \item[ii)] If $\mathbf{X} = \mathbf{Y}(\mathbf{I}_{m}+\mathbf{Y}'\mathbf{Y})^{-1/2}$, we have
     \begin{equation}\label{eq3}
        (d\mathbf{X}) = \left|\mathbf{I}_{m} + \mathbf{Y}'\mathbf{Y}\right|^{-(n+m+1)/2} (d\mathbf{Y}),
     \end{equation}
\end{description}
where $n \geq m$.
\end{theorem}
\begin{proof}
\textbf{i}) Define $\mathbf{A} = \mathbf{Y}'\mathbf{Y} =
(\mathbf{I}_{m}-\mathbf{X}'\mathbf{X})^{-1/2}\mathbf{X}'\mathbf{X}(\mathbf{I}_{m}-\mathbf{X}'\mathbf{X})^{-1/2}$
and $\mathbf{B} = \mathbf{X}'\mathbf{X}$, then by \citet[Theorem 2.1.14, p. 66]{mh:05}, for
$\mathbf{H}_{1},\mathbf{G}_{1} \in \mathcal{V}_{m,n}$
\begin{eqnarray*}
  (d\mathbf{Y}) &=& 2^{-m}|\mathbf{A}|^{(n-m-1)/2}(d\mathbf{A})(\mathbf{H}'_{1}d\mathbf{H}_{1}) \\
  (d\mathbf{X}) &=& 2^{-m}|\mathbf{B}|^{(n-m-1)/2}(d\mathbf{B})(\mathbf{G}'_{1}d\mathbf{G}_{1}),
\end{eqnarray*}
thus
\begin{eqnarray}
  \label{eq4}(d\mathbf{A}) &=& 2^{m}|\mathbf{A}|^{-(n-m-1)/2}(d\mathbf{Y})(\mathbf{H}'_{1}d\mathbf{H}_{1})^{-1} \\
  \label{eq5}(d\mathbf{B}) &=& 2^{m}|\mathbf{B}|^{-(n-m-1)/2}(d\mathbf{X})(\mathbf{G}'_{1}d\mathbf{G}_{1})^{-1}.
\end{eqnarray}
Also, observe that
\begin{eqnarray*}
  \mathbf{A} &=& (\mathbf{I}_{m}-\mathbf{X}'\mathbf{X})^{-1/2}\mathbf{X}'\mathbf{X}(\mathbf{I}_{m}-\mathbf{X}'
        \mathbf{X})^{-1/2}\\
    &=& (\mathbf{I}_{m}-\mathbf{B})^{-1/2}\mathbf{B}(\mathbf{I}_{m}-\mathbf{B})^{-1/2} \\
    &=& \left[\mathbf{B}^{-1}(\mathbf{I}_{m} - \mathbf{B})\right]^{-1/2} \left[\mathbf{B}^{-1}(\mathbf{I}_{m}
        - \mathbf{B})\right]^{-1/2}\\
    &=& \left(\mathbf{B}^{-1}- \mathbf{I}_{m}\right)^{-1/2} \left(\mathbf{B}^{-1}- \mathbf{I}_{m} \right)^{-1/2}\\
    &=& \left(\mathbf{B}^{-1}- \mathbf{I}_{m}\right)^{-1} = (\mathbf{I}_{m} - \mathbf{B})^{-1}\mathbf{B} =
    (\mathbf{I}_{m} - \mathbf{B})^{-1}-\mathbf{I}_{m}.
\end{eqnarray*}
Hence, from \citet[Theorem 2.1.8, p. 59]{mh:05}, we have that
\begin{equation}\label{eq6}
    (d\mathbf{A}) = |\mathbf{I}_{m} - \mathbf{B}|^{-(m+1)}(d\mathbf{B}).
\end{equation}
Substituting (\ref{eq6}) into (\ref{eq5}) and equating to (\ref{eq4}) we have
$$
  2^{m}|\mathbf{A}|^{-(n-m-1)/2}(d\mathbf{Y})(\mathbf{H}'_{1}d\mathbf{H}_{1})^{-1}
  \hspace{6cm}
$$
$$
  \hspace{4cm} = 2^{m}|\mathbf{I}_{m} -
  \mathbf{B}|^{-(m+1)}|\mathbf{B}|^{-(n-m-1)/2}(d\mathbf{X})(\mathbf{G}'_{1}d\mathbf{G}_{1})^{-1}.
$$
Then, by the uniqueness of the nonnormalised measure on Stiefel manifold,
$(\mathbf{H}'_{1}d\mathbf{H}_{1}) = (\mathbf{G}'_{1}d\mathbf{G}_{1})$. Thus,
$$
  (d\mathbf{Y}) = |\mathbf{A}|^{(n-m-1)/2}|\mathbf{I}_{m} - \mathbf{B}|^{-(m+1)}|\mathbf{B}|^{-(n-m-1)/2}(d\mathbf{X}).
$$
And using $|\mathbf{A}| = |(\mathbf{I}_{m}-\mathbf{X}'\mathbf{X})^{-1/2} \mathbf{X}'\mathbf{X}
(\mathbf{I}_{m}-\mathbf{X}'\mathbf{X})^{-1/2}| = |(\mathbf{I}_{m}-\mathbf{X}'\mathbf{X})|^{-1}
|\mathbf{X}'\mathbf{X}|$ and $|\mathbf{B}| = |\mathbf{X}'\mathbf{X}|$,  the required result is
obtained.

\textbf{ii}). The proof is similar to the preceding exposition given in \textbf{i}).
\end{proof}

Now, $\mathbf{Z}$ or order $N \times m$ has a matrix variate elliptical distribution with respect to the
Lebesgue measure $(d\mathbf{Z})$, if its density function is given by
\begin{equation}\label{elliptical}
 dF_{_{\mathbf{Z}}} (\mathbf{Z}) \propto |\mathbf{\Sigma}|^{-N/2}|\mathbf{\Theta}|^{-m/2}
  h\left(\tr\left((\mathbf{Z}-\boldsymbol{\mu})^{T}\mathbf{\Sigma}^{-1}(\mathbf{Z}-
  \boldsymbol{\mu})\mathbf{\Theta}^{-1}\right)\right)(d\mathbf{Z}).
\end{equation}
The location  parameter is $\boldsymbol{\mu}:N \times m$;  and the scale parameters are given by the
positive definite matrices $\mathbf{\Sigma}:N \times N$ and $\mathbf{\Theta}:m \times m$. The
distribution is denoted by $\mathbf{Z}\sim \mathcal{E}_{N \times m}(\boldsymbol{\mu} ,\mathbf{\Sigma},
\mathbf{\Theta}; h)$, and indexed by the kernel function $h\mbox{: } \Re \to [0,\infty)$, where
$\int_{0}^\infty u^{Nm-1}h(u^2)\text{d}u < \infty$.

A special case of a matrix variate elliptical distribution appears when $\boldsymbol{\mu} =
\mathbf{0}$, $\mathbf{\Sigma} = \mathbf{I}_{N}$ and $\mathbf{\Theta} = \mathbf{I}_{m}$. In that
situation we say that $\mathbf{Z}$ has a matrix variate spherical distribution.

\section{Multimatricvariate distributions}\label{sec:3}

First, start with $\mathbf{X}\sim \mathcal{E}_{N \times m}(\mathbf{0}, \mathbf{I}_{N},
\mathbf{I}_{m}; h)$, such that $n_{1}+ \cdots + n_{k} = N$, and $\mathbf{X} = \left(\mathbf{X}'_{1},
\dots, \mathbf{X}'_{k} \right)'$. Then (\ref{elliptical}) can be written as
$$
  dF_{\mathbf{X}_{1}, \dots,\mathbf{X}_{k}}(\mathbf{X}_{1}, \dots,\mathbf{X}_{k}) =
  h(\tr(\mathbf{X}'_{1}\mathbf{X}_{1}+\cdot+\mathbf{X}'_{k}\mathbf{X}_{k}))\bigwedge_{i=1}^{k}(d\mathbf{X}_{i}),
$$
where $\mathbf{X}_{i} \in \Re^{n_{i} \times m}$, $n_{i} \geq m$, $i = 1, \dots, k$. Note that the random
matrices $\mathbf{X}_{1}, \dots,\mathbf{X}_{k}$ are probabilistically dependent.  Only under a matrix
variate normal distribution, the random matrices are independent.

Now, define $\mathbf{W}_{i} = \mathbf{X}'_{i}\mathbf{X}_{i}$, $i = 1, \dots, k$, therefore from
\citet[Theorem 2.1.14, p. 66]{mh:05}
$$
  (d\mathbf{X}_{i}) = 2^{-m}|\mathbf{W}_{i}|^{(n_{i}-m-1)/2}(d\mathbf{W}_{i})(\mathbf{H}_{i}d\mathbf{H}_{i}),
$$
where $\mathbf{H}_{i} \in \mathcal{V}_{m,n_{i}} = \{\mathbf{H}_{i} \in
\Re^{n_{i},m}|\mathbf{H}'_{i}\mathbf{H}_{i} = \mathbf{I}_{m}\}$, $\mathcal{V}_{m,n_{i}}$ is termed the
Stiefel manifold. Hence
$$
  \bigwedge_{i=1}^{k}(d\mathbf{X}_{i}) = 2^{-mk} \prod_{i=1}^{k}|\mathbf{W}_{i}|^{(n_{i}-m-1)/2}
  \bigwedge_{i=1}^{k}(d\mathbf{W}_{i}) \bigwedge_{i=1}^{k}(\mathbf{H}_{i}d\mathbf{H}_{i}),
$$
and by \citet[Theorem 2.1.15, p. 70]{mh:05}
$$
  \int_{\mathbf{H}_{1}\in \mathcal{V}_{m,n_{1}}} \cdots \int_{\mathbf{H}_{k} \in \mathcal{V}_{m,n_{k}}}
  \bigwedge_{i=1}^{k}(\mathbf{H}_{i}d\mathbf{H}_{i}) = \prod_{i=1}^{k} \int_{\mathbf{H}_{i} \in
  \mathcal{V}_{m,n_{i}}} (\mathbf{H}_{i}d\mathbf{H}_{i}) = \frac{2^{km} \pi^{Nm/2}}{\displaystyle\prod_{i=1}^{k}
  \Gamma_{m}[n_{i}/2]}.
$$
Thus the joint density $dF_{\mathbf{W}_{1}, \dots,\mathbf{W}_{k}}(\mathbf{W}_{1},
\dots,\mathbf{W}_{k})$ is given by
\begin{equation}\label{GW}
     \pi^{Nm/2} \left(\prod_{i=1}^{k} \frac{|\mathbf{W}_{i}|^{(n_{i}-m-1)/2}}{\Gamma_{m}[n_{i}/2]}\right )
    h(\tr(\mathbf{W}_{1}+ \cdots + \mathbf{W}_{k}))\bigwedge_{i=1}^{k}(d\mathbf{W}_{i}).
\end{equation}
This distribution was studied in \citet{fz:90} and it will be called \emph{multimatricvariate generalised
Wishart distribution}. Now, let $\mathbf{\Sigma}_{i} \in \Re^{n_{i} \times n_{i}}$ be positive definite
matrices, with $i = 1,\dots,k$, and define $\mathbf{V}_{i} = \mathbf{\Sigma}_{i}^{1/2} \mathbf{W}_{i}
\mathbf{\Sigma}_{i}^{1/2}$, where $\mathbf{\Sigma}_{i}^{1/2}$ is the positive definite square root of
$\mathbf{\Sigma}_{i}$, i.e., $\mathbf{\Sigma}_{i} = \mathbf{\Sigma}_{i}^{1/2} \mathbf{\Sigma}_{i}^{1/2}$,
see \citet[Theorem A9.3, p. 588]{mh:05}. Hence, from \citet[Theorem 2.1.6, p. 58]{mh:05}
$(d\mathbf{W}_{i}) = |\mathbf{\Sigma}_{i}|^{-(m-1)/2}(d\mathbf{V}_{i})$. And
$$
  \bigwedge_{i=1}^{k}(d\mathbf{W}_{i}) = \prod_{i=1}^{k} |\mathbf{\Sigma}_{i}|^{-(m-1)/2}\bigwedge_{i=1}^{k}
  (d\mathbf{V}_{i}).
$$
Then, by (\ref{GW}), the joint density $dF_{\mathbf{V}_{1}, \dots,\mathbf{V}_{k}}(\mathbf{V}_{1},
\dots,\mathbf{V}_{k})$ is given by
\begin{equation}\label{mGW}
     \pi^{Nm/2} \left(\prod_{i=1}^{k} \frac{|\mathbf{V}_{i}|^{(n_{i}-m-1)/2}}{\Gamma_{m}[n_{i}/2]
    |\mathbf{\Sigma}_{i}|^{n_{i}/2}}\right ) h(\tr(\mathbf{\Sigma}_{1}^{-1}\mathbf{V}_{1}+ \cdots
    + \mathbf{\Sigma}_{k}^{-1}\mathbf{V}_{k})) \bigwedge_{i=1}^{k}(d\mathbf{V}_{i}).
\end{equation}
A simple particular case is the normal bivariate distribution, full derived by \citet{n:07}.

Now, a generalization of the bivariate gama-beta distribution of \citet{n:13} is provided next in the
context of a multimatricvariate distribution under an elliptical model:

\begin{theorem}\label{teo2} Assume that $\mathbf{X} = \left(\mathbf{X}'_{0}, \dots,
\mathbf{X}'_{k} \right)'$ has  a matrix variate spherical distribution, with $\mathbf{X}_{i} \in
\Re^{n_{i} \times m}$, $n_{i} \geq m$, $i = 0,1, \dots, k$. Define $\mathbf{V}_{0} =
\mathbf{X}'_{0}\mathbf{X}_{0}$ and $\mathbf{T}_{i} = \mathbf{X}_{i} \mathbf{V}_{0}^{-1/2}$, $i =
1,\dots,k$.\\ Then, the joint density   $dF_{\mathbf{V}_{0},\mathbf{T}_{1},
\dots,\mathbf{T}_{k}}(\mathbf{V}_{0},\mathbf{T}_{1}, \dots,\mathbf{T}_{k})$ is given by
\begin{equation}\label{mwt}
   \frac{\pi^{n_{0}m/2}}{\Gamma_{m}[n_{0}/2]}|\mathbf{V}_{0}|^{\left(n^{*}-m-1\right)/2}
  h\left[\tr \mathbf{V}_{0}\left(\mathbf{I}_{m}+\displaystyle\sum_{i=1}^{k}\mathbf{T}'_{i}\mathbf{T}_{i}\right)\right]
  (d\mathbf{V}_{0})\bigwedge_{i=1}^{k}\left(d\mathbf{T}_{i}\right),
\end{equation}
where $n^{*} = n_{0}+n_{1}+\cdots+n_{k}$, $\mathbf{V}_{0} > 0$ and $\mathbf{T}_{i} \in \Re^{n_{i} \times
m}$, $i = 1,\dots,k$. This distribution will be termed \emph{multimatricvariate Wishart-T distribution}.
Moreover, the so called  \emph{multimatricvariate T distribution} is the marginal density
$dF_{\mathbf{T}_{1}, \dots,\mathbf{T}_{k}}(\mathbf{T}_{1}, \dots,\mathbf{T}_{k})$ of $\mathbf{T}_{1},
\dots,\mathbf{T}_{k}$ and is given by
\begin{equation}\label{mt}
   \frac{\Gamma_{m}[n^{*}/2]}{\pi^{m(n^{*}-n_{0})/2}\Gamma_{m}[n_{0}/2]}
  \left|\mathbf{I}_{m}+\displaystyle\sum_{i=1}^{k}\mathbf{T}'_{i}\mathbf{T}_{i}\right|^{-n^{*}/2}
  \bigwedge_{i=1}^{k}\left(d\mathbf{T}_{i}\right).
\end{equation}
\end{theorem}
\begin{proof}
The join density of $\mathbf{X}_{0}, \dots, \mathbf{X}_{k}$ is
$$
  dF_{\mathbf{X}_{0}, \dots,\mathbf{X}_{k}}(\mathbf{X}_{1}, \dots,\mathbf{X}_{k}) =
  h(\tr(\mathbf{X}'_{0}\mathbf{X}_{0}+\cdot+\mathbf{X}'_{k}\mathbf{X}_{k}))\bigwedge_{i=0}^{k}(d\mathbf{X}_{i}).
$$
Then, by setting $\mathbf{V}_{0} = \mathbf{X}'_{0}\mathbf{X}_{0}$ and $\mathbf{T}_{i} =
\mathbf{X}_{i}\mathbf{V}_{0}^{-1/2}$, $i = 1,\dots,k$, we have by \citet[theorems 2.1.4 and
2.1.14]{mh:05}.
$$
  \bigwedge_{i=0}^{k}(d\mathbf{X}_{i}) = 2^{-m} |\mathbf{V}_{0}|^{(n_{0}-m-1)/2}(d\mathbf{V}_{0})\wedge
  (\mathbf{H}'_{1}d\mathbf{H}_{1}) \cdot \left (\prod_{i=1}^{k} |\mathbf{V}_{0}|^{n_{i}/2}\right)
  \bigwedge_{i=1}^{k}(d\mathbf{T}_{i}),
$$
where $\mathbf{H}_{1} \in \mathcal{V}_{m,n_{0}}$. Thus the joint density $dF_{\mathbf{V}_{0},
\mathbf{H}_{1}, \mathbf{T}_{1}, \dots,\mathbf{T}_{k}}(\mathbf{V}_{0},\mathbf{H}_{1},\mathbf{T}_{1},
\dots,\mathbf{T}_{k})$ is
$$
  |\mathbf{V}_{0}|^{\left(n^{*}-m-1\right)/2}
  h\left(\tr \mathbf{V}_{0}\left(\mathbf{I}_{m}+\displaystyle\sum_{i=1}^{k}\mathbf{T}'_{i}\mathbf{T}_{i}\right)\right)
  (d\mathbf{V}_{0})\bigwedge_{i=1}^{k}\left(d\mathbf{T}_{i}\right)\wedge
  (\mathbf{H}'_{1}d\mathbf{H}_{1}).
$$
where $n^{*} = n_{0}+ n_{1}+\cdots+n_{k}$.

The integration over the Stiefel manifold, by using \citet[Theorem 2.1.15]{mh:05}, provides the required
marginal density (\ref{mwt}) of $\mathbf{V}_{0},\mathbf{T}_{1}, \dots,\mathbf{T}_{k}$.

Now, the integration of (\ref{mwt}) with respect to $\mathbf{V}_{0}$, by using \citet[Corollary 5.1.5.1,
p.169]{gv:93} or \citet[Corollary 2, p. 34]{cdg:10}, gives the marginal density of $\mathbf{T}_{1},
\dots,\mathbf{T}_{k}$. Then, the marginal density  is invariant under the elliptical family  and is given
by (\ref{mt}).
\end{proof}

Now, taking $\mathbf{F}_{i} = \mathbf{T}'_{i}\mathbf{T}_{i}$, $i = 1,
\dots,k$ in Theorem \ref{teo2}, we obtain the following multimatricvariate distributions.

\begin{theorem}\label{teo3}
Under the  hypotheses of Theorem \ref{teo2}, define $\mathbf{F}_{i} = \mathbf{T}'_{i}\mathbf{T}_{i}$, $i
= 1, \dots,k$. Then the joint density $dF_{\mathbf{V}_{0}, \mathbf{F}_{1},
\dots,\mathbf{F}_{k}}(\mathbf{V}_{0},\mathbf{F}_{1}, \dots,\mathbf{F}_{k})$ is
$$
   \frac{\pi^{m n^{*}/2}}{\displaystyle\prod_{i=0}^{k}\Gamma_{m}[n_{i}/2]}
  |\mathbf{V}_{0}|^{\left(n^{*}-m-1\right)/2} \prod_{i=1}^{k}|\mathbf{F}_{i}|^{(n_{i}-m-1)/2}
  h\left(\tr \mathbf{V}_{0}\left(\mathbf{I}_{m}+\displaystyle\sum_{i=1}^{k}\mathbf{F}_{i}\right)\right)
$$
$$\hspace{10cm}
  (d\mathbf{V}_{0})\bigwedge_{i=1}^{k}\left(d\mathbf{F}_{i}\right).
$$
This distribution will be termed \emph{multimatricvariate Wishart-beta type II distribution}.  The
associated marginal  density function
 $dF_{\mathbf{F}_{1}, \dots, \mathbf{F}_{k}}(\mathbf{F}_{1}, \dots,\mathbf{F}_{k})$
is given by
\begin{equation}\label{MMVF}
     \frac{\Gamma_{m}[n^{*}/2]}{\displaystyle\prod_{i=0}^{k}\Gamma_{m}[n_{i}/2]}
  \prod_{i=1}^{k}|\mathbf{F}_{i}|^{(n_{i}-m-1)/2}
  \left|\mathbf{I}_{m}+\displaystyle\sum_{i=1}^{k}\mathbf{F}_{i}\right|^{-n^{*}/2}
  \bigwedge_{i=1}^{k}\left(d\mathbf{F}_{i}\right).
\end{equation}
This distribution will be called \emph{multimatricvariate beta type II distribution}.
\end{theorem}
\begin{proof}
Defining $\mathbf{F}_{i} = \mathbf{T}'_{i}\mathbf{T}_{i}$, $i = 1, \dots,n$, and applying \citet[Theorem
2.1.14]{mh:05}, we obtain
$$
  \bigwedge_{i=1}^{k}(d\mathbf{T}_{i}) = 2^{-mk} \prod_{i=1}^{k}|\mathbf{F}_{i}|^{(n_{i}-m-1)/2}
  \bigwedge_{i=1}^{k}(d\mathbf{F}_{i}) \bigwedge_{i=1}^{k}(\mathbf{H}_{i}d\mathbf{H}_{i}),
$$
where $\mathbf{H}_{i} \in \mathcal{V}_{m,n_{i}}$, $i = 1, \dots,k$. Then, integrating iteratively with
respect to $\mathbf{H}_{i} \in \mathcal{V}_{m,n_{i}}$, the results are obtained form (\ref{mwt}) and
(\ref{mt}), respectively, see \citet[Theorem 2.1.15, p. 70]{mh:05}.
\end{proof}

Now, consider the  Gaussian kernel $h(y)=(2\pi)^{-n^{*}m/2}e^{-y/2}$ in Theorem \ref{teo3}, then  a
number of published results for bivariate and bimatrix cases are obtained. But they were originally full
derived under a context of independence, see \citet{dggj:10a}, \citet{e:11} and \citet{n:13}. Using a
different approach, the normal case of (\ref{MMVF}) was derived by \citet[Theorem 3.1]{or:64}.

In a similar way, we can obtain the so called multimatricvariate Wishart-Pearson type II, Pearson type
II, Wishart-beta type I and beta type I distributions. The results are based on the theorems \ref{teo2}
and \ref{teo3} using the Jacobians of Theorem \ref{teo1}. The following result summarizes our approach.

\begin{theorem}\label{teo4}
Suppose that $\mathbf{X} = \left(\mathbf{X}'_{0}, \dots, \mathbf{X}'_{k} \right)'$ has a matrix variate
spherical distribution, with $\mathbf{X}_{i} \in \Re^{n_{i} \times m}$, $n_{i} \geq m$, $i = 0,1, \dots,
k$. Define $\mathbf{W}_{0} = \mathbf{X}'_{0}\mathbf{X}_{0}$ and $\mathbf{R}_{i} =
\mathbf{X}_{i}(\mathbf{W}_{0}+\mathbf{X}_{i}'\mathbf{X}_{i})^{-1/2}$, $i = 1,\dots,k$. Then the joint
density of $\mathbf{W}_{0},\mathbf{R}_{1}, \dots,\mathbf{R}_{k}$, denoted as
$dF_{\mathbf{W}_{0},\mathbf{R}_{1}, \dots,\mathbf{R}_{k}}(\mathbf{W}_{0},\mathbf{R}_{1},
\dots,\mathbf{R}_{k})$, is given by
\begin{equation}\label{mwP2}
\frac{\pi^{n_{0}m/2}}{\Gamma_{m}[n_{0}/2]}|\mathbf{W}_{0}|^{\left(n^{*}-m-1\right)/2} h\left(\tr
\mathbf{W}_{0}\left(\mathbf{I}_{m}+\displaystyle\sum_{i=1}^{k}(\mathbf{I}_{m}-
\mathbf{R}'_{i}\mathbf{R}_{i})^{-1}\mathbf{R}'_{i}\mathbf{R}_{i}\right)\right)
$$
$$
  \times  \prod_{i=i}^{k}|\mathbf{I}_{m}-
  \mathbf{R}'_{i}\mathbf{R}_{i}|^{-(n_{i}+m+1)/2}
  (d\mathbf{W}_{0})\bigwedge_{i=1}^{k}\left(d\mathbf{R}_{i}\right),
\end{equation}
where $n^{*} = n_{0}+n_{1}+\cdots+n_{k}$, $\mathbf{W}_{0} > 0$ and $\mathbf{R}_{i} \in \Re^{n_{i} \times
m}$, $i = 1,\dots,k$. This distribution will be termed \emph{multimatricvariate Wishart-Pearson type II
distribution}. Moreover,  the marginal density $dF_{\mathbf{R}_{1}, \dots,\mathbf{R}_{k}}(\mathbf{R}_{1},
\dots,\mathbf{R}_{k})$ is
\begin{equation}\label{mP2}
   \frac{\Gamma_{m}[n^{*}/2]}{\pi^{m (n^{*}-n_{0})/2}\Gamma_{m}[n_{0}/2]}
  \left|\prod_{i=1}^{k}(\mathbf{I}_{m}-\mathbf{R}'_{i}\mathbf{R}_{i})+\displaystyle\sum_{i=1}^{k}\prod_{\build{}{j=1}{j \neq i}}^{k}
  (\mathbf{I}_{m}-\mathbf{R}'_{j}\mathbf{R}_{j})\mathbf{R}'_{i}\mathbf{R}_{i}\right|^{-n^{*}/2}
$$
$$
  \times \prod_{i=1}^{k}\left|\mathbf{I}_{m}-\mathbf{R}'_{i}\mathbf{R}_{i}\right|^{\left(\sum_{\build{}{j=1}{j\neq i}}^{k}n_{j}-m-1\right)/2}
  \bigwedge_{i=1}^{k}\left(d\mathbf{R}_{i}\right),
\end{equation}
which will  be called \emph{multimatricvariate Pearson type II distribution}.
\end{theorem}
\begin{proof}
First, if $\mathbf{T}$ has a matricvariate $T$ distribution, then
$\mathbf{R}=\mathbf{T}(\mathbf{I}_{m}+\mathbf{T}'\mathbf{T})^{-1/2}$ has a matricvariate Pearson type II
distribution, see  \citet{di:67}. Now, using  Theorem \ref{teo1}, we have
$$
  \bigwedge_{i=1}^{k}(d\mathbf{T}) = \prod_{i=1}^{k}\left|\mathbf{I}_{m} - \mathbf{R}'_{i}
  \mathbf{R}_{i}\right|^{-(n_{i}+m+1)/2} \bigwedge_{i=1}^{k}(d\mathbf{R}_{i}).
$$
Then, (\ref{mwP2}) follows is a consequence of (\ref{mwt}),
$\mathbf{T}=\mathbf{R}(\mathbf{I}_{m}-\mathbf{R}'\mathbf{R})^{-1/2}$ and
$$
  \left|\mathbf{I}_{m}+\displaystyle\sum_{i=1}^{k}\mathbf{T}'_{i}\mathbf{T}_{i}\right| =
  \frac{\left|\displaystyle \prod_{i=1}^{k}(\mathbf{I}_{m}-\mathbf{R}'_{i}\mathbf{R}_{i})+\sum_{i=1}^{k}
  \prod_{\build{}{j=1}{j \neq i}}^{k} (\mathbf{I}_{m}-\mathbf{R}'_{j}\mathbf{R}_{j})\mathbf{R}'_{i}
  \mathbf{R}_{i}\right|}{\displaystyle\prod_{i=1}^{k}\left|\mathbf{I}_{m}-\mathbf{R}'_{i}\mathbf{R}_{i}\right|}.
$$
Finally, by \citet[Corollary 5.1.5.1, p. 169]{gv:93} or \citet[Corollary 2, p. 34]{cdg:10} for the
marginal density of $\mathbf{R}_{1}, \dots,\mathbf{R}_{k}$ we have the required result (\ref{mP2}).
\end{proof}

The approach of Theorem \ref{teo3} can be applied under the transformation $\mathbf{U}_{i} =
\mathbf{R}'_{i}\mathbf{R}_{i}$, $i = 1, \dots,k$. First,  by \citet[Theorem 2.1.14, p. 66]{mh:05}, we
have
$$
  \bigwedge_{i=1}^{k}(d\mathbf{R}_{i}) = 2^{-mk} \prod_{i=1}^{k}|\mathbf{U}_{i}|^{(n_{i}-m-1)/2}
  \bigwedge_{i=1}^{k}(d\mathbf{U}_{i}) \bigwedge_{i=1}^{k}(\mathbf{P}_{i}d\mathbf{P}_{i}),
$$
where $\mathbf{P}_{i} \in \mathcal{V}_{m,n_{i}}$, $i = 1, \dots,n$. Then,  multiple integration with
respect to $\mathbf{P}_{i} \in \mathcal{V}_{m,n_{i}}$, and expressions (\ref{mwP2}) and (\ref{mP2})
provide the following multimatricvariate distributions.

\begin{theorem}\label{teo5}
Consider $\mathbf{U}_{i} = \mathbf{R}'_{i}\mathbf{R}_{i}$ with $i = 1, \dots,k$, in Theorem \ref{teo4}
then the joint density $dF_{\mathbf{W}_{0}, \mathbf{U}_{1},\dots,\mathbf{U}_{k}}(\mathbf{W}_{0},
\mathbf{U}_{1}, \dots,\mathbf{U}_{k})$ is
$$
   \frac{\pi^{m n^{*}/2}}{\displaystyle\prod_{i=0}^{k}\Gamma_{m}[n_{i}/2]}
  |\mathbf{W}_{0}|^{\left(n^{*}-m-1\right)/2} \prod_{i=1}^{k}\left(\frac{|\mathbf{U}_{i}|}{|\mathbf{I}_{m} -
  \mathbf{U}_{i}|}\right)^{(n_{i}-m-1)/2} \hspace{2cm}
$$
$$\hspace{2cm}
  \times h\left(\tr \mathbf{W}_{0}\left(\mathbf{I}_{m}+\displaystyle\sum_{i=1}^{k}(\mathbf{I}_{m} -
  \mathbf{U}_{i})^{-1}\mathbf{U}_{i}\right)\right)
  (d\mathbf{W}_{0})\bigwedge_{i=1}^{k}\left(d\mathbf{U}_{i}\right).
$$
This distribution will be termed \emph{multimatricvariate Wishart-beta type I distribution}.  Moreover,
the density function $dF_{\mathbf{U}_{1}, \dots, \mathbf{U}_{k}}(\mathbf{U}_{1}, \dots,\mathbf{U}_{k})$
is  written as
\begin{equation}\label{MMB1}
   \frac{\Gamma_{m}[n^{*}/2]}{\displaystyle\prod_{i=0}^{k}\Gamma_{m}[n_{i}/2]} \left|\displaystyle\prod_{i=1}^{k}
  (\mathbf{I}_{m} - \mathbf{U}_{i}) + \sum_{i=1}^{k}\prod_{\build{}{j=1}{j\neq i}}^{k}(\mathbf{I}_{m} -
  \mathbf{U}_{j})\mathbf{U}_{i}\right|^{-n^{*}/2} \hspace{3cm}
$$
$$ \hspace{2cm}
  \prod_{i=1}^{k} \left(|\mathbf{U}_{i}|^{(n_{i}-m-1)/2}||\mathbf{I}_{m} - \mathbf{U}_{i}|^{\left(
  \sum_{\build{}{j=1}{j\neq i}}^{k}n_{j}-m-1\right)/2}\right) \bigwedge_{i=1}^{k}\left(d\mathbf{U}_{i}\right).
\end{equation}
This marginal distribution will be called \emph{multimatricvariate beta type I distribution}.
\end{theorem}

Observe that Theorem \ref{teo5} can be derived from Theorem \ref{teo3} by taking $\mathbf{F} = (\mathbf{I}_{m} -
\mathbf{U})^{-1} - \mathbf{I}_{m}$, where
$\mathbf{F}$ has a matricvariate beta type II distribution and $\mathbf{U}$  has a matricvariate beta type I distribution.

\begin{remark}\label{rem1}
Note that the parameter domain of the  multimatricvariate generalised Wishart distribution can be
extended. This distribution is also valid when $n_{i}/2$ are replaced by $a_{i}$, $i = 1,\dots,k$, where
$N/2 = a = a_{1}+ \cdots+a_{k}$ and the $a_{i}$s are complex numbers with positive real part. In this
scenario, it will be called \emph{multimatricvariate generalised Gamma distribution}. The referred
extended domain holds for all the multimatricvariate distributions considered in this article, but the
geometric and/or statistical interpretation is difficult. For example, the multimatricvariate
Wishart-beta type I distribution can be termed as \emph{multimatricvariate gamma-beta type I
distribution}. We emphasize that the derived expressions correspond to  families of distributions indexed
by kernel functions $h(\cdot)$ with potential advantage for multiple real situations. We most note that
the extended domain hus losing the geometric and/or statistics interpretation of the parameters.
\end{remark}

\section{Some properties}\label{sec:4}

In this section we study several  properties of the multimatricvariate distributions. They include the
marginal distributions and the inverse distributions, among many others aspects.

The construction of the multimatricvariate distributions gives us a simple way of computing the
associated marginal distributions. For example if the multimatricvariate
$$
  (\mathbf{V}_{0},\mathbf{R}_{1}',\cdots, \mathbf{R}_{k}')'
$$
has a gamma-Pearson type II distribution, then,
$\mathbf{V}_{0}$ has a matricvariate gamma distribution and the $\mathbf{R}_{i}$s have a matricvariate
Pearson type II distributions for all $i = 1, \dots,n$.

Now, if $k = 2$ and the extended parameters  $n_{i}/2 = a_{i}$, $i = 0,1,2$ are considered in
(\ref{MMB1}), we have that  $n^{*}/2 = a_{0}+ a_{1}+a_{2}$ and
$$
  \left|\displaystyle\prod_{i=1}^{k}(\mathbf{I}_{m} - \mathbf{U}_{i}) + \sum_{i=1}^{k}\prod_{\build{}{j=1}
  {j\neq i}}^{k}(\mathbf{I}_{m} - \mathbf{U}_{j})\mathbf{U}_{i}\right|\hspace{7cm}
$$
$$
  = |(\mathbf{I}_{m} - \mathbf{U}_{1})(\mathbf{I}_{m}
  - \mathbf{U}_{2})+(\mathbf{I}_{m} - \mathbf{U}_{2})\mathbf{U}_{1}+(\mathbf{I}_{m} -
  \mathbf{U}_{1})\mathbf{U}_{2}| =|\mathbf{I}_{m} - \mathbf{U}_{1}\mathbf{U}_{2}|.
$$
Thus
$$
  dF_{\mathbf{U}_{1},\mathbf{U}_{2}}(\mathbf{U}_{1},\mathbf{U}_{2}) = \frac{\Gamma_{m}[a_{0}+ a_{1}+a_{2}]}
  {\Gamma_{m}[a_{0}]\Gamma_{m}[a_{1}] \Gamma_{m}[a_{2}]}|\mathbf{U}_{1}|^{a_{1}-(m+1)/2} |\mathbf{U}_{2}|^{a_{2}-(m+1)/2}
  \hspace{2cm}
$$
$$\hspace{2cm}
  \frac{|\mathbf{I}_{m} - \mathbf{U}_{1}|^{a_{0}+a_{2}-(m+1)/2}|\mathbf{I}_{m} - \mathbf{U}_{2}|^{a_{0}+a_{1}-(m+1)/2}}
  {|\mathbf{I}_{m} - \mathbf{U}_{1}\mathbf{U}_{2}|^{a_{_{0}}+ a_{1}+a_{2}}}
  (d\mathbf{U}_{1})\wedge(d\mathbf{U}_{2}).
$$
This is the bimatrix variate beta type I distribution obtained  by \cite{dggj:10a, dggj:10b} in the real
and complex cases (see also \citet{e:11}). The distribution was derived with an approach based on certain
independent variables, instead of the dependent scenario of our method. However, both distributions must
coincide given the known invariance property of the elliptical distributions, see \citep[Section 5.3, p.
182]{gv:93}.

In fact, we can derive a number of multimatricvariate distributions mixing three of more  types of joint densities.

\begin{theorem}\label{teo6}
Suppose that $\mathbf{X} = \left(\mathbf{X}'_{0}, \mathbf{X}'_{1}, \mathbf{X}'_{2} \right)'$ has a matrix
variate spherical distribution, with $\mathbf{X}_{i} \in \Re^{n_{i} \times m}$, $n_{i} \geq m$, $i =
0,1,2$. Define $\mathbf{W}_{0} = \mathbf{X}'_{0}\mathbf{X}_{0}$, $\mathbf{T} =
\mathbf{X}_{1}\mathbf{W}_{0}^{-1/2}$, $\mathbf{R}= \mathbf{X}_{2}(\mathbf{W}_{0} +
\mathbf{X}_{2}'\mathbf{X}_{2})^{-1/2}$ and $\mathbf{W} = \mathbf{W}_{0}+\mathbf{X}'_{2}\mathbf{X}_{2}$.
Then the joint density $dF_{\mathbf{W},\mathbf{T},\mathbf{R}}(\mathbf{W},\mathbf{T},\mathbf{R})$ is given
by
\begin{equation}\label{mwtP2}
   \frac{\pi^{n_{0}m/2}}{\Gamma_{m}[n_{0}/2]}|\mathbf{W}|^{\left(n^{*}-m-1\right)/2}
  h\left(\tr \mathbf{W}+ (\mathbf{I}_{m}- \mathbf{R}'\mathbf{R})\mathbf{W}^{1/2} \mathbf{T}'\mathbf{TW}^{1/2}\right)
$$
$$
  \hspace{2cm}\times  |\mathbf{I}_{m}- \mathbf{R}'\mathbf{R}|^{(n_{0}+n_{1}-m-1)/2}
  (d\mathbf{W})\wedge(d\mathbf{T})\wedge \left(d\mathbf{R}\right),
\end{equation}
where $n^{*} = n_{0}+n_{1}+n_{2}$, $\mathbf{W} > 0$, $\mathbf{T} \in \Re^{n_{1} \times m}$ and
$\mathbf{R} \in \Re^{n_{2} \times m}$. The distribution of $\mathbf{W}_{0},\mathbf{T},\mathbf{R}$ will be
termed \emph{trimatricvariate Wishart-T-Pearson type II distribution}. Moreover, setting  $\mathbf{F} =
\mathbf{T}'\mathbf{T}$ and $\mathbf{U} = \mathbf{R}'\mathbf{R}$, the density $dF_{\mathbf{W},\mathbf{F},
\mathbf{U}}(\mathbf{W}, \mathbf{F}, \mathbf{U})$ is
\begin{equation}\label{mwb1b2}
   \frac{\pi^{n^{*}m/2}|\mathbf{W}|^{\left(n^{*}-m-1\right)/2}}{\Gamma_{m}[n_{0}/2]\Gamma_{m}[n_{1}/2]\Gamma_{m}[n_{2}/2]}
  h\left(\tr \mathbf{W}+ (\mathbf{I}_{m}- \mathbf{U})\mathbf{W}^{1/2}
  \mathbf{FW}^{1/2}\right) \hspace{2cm}
$$
$$
  \hspace{2mm}\times  |\mathbf{I}_{m}- \mathbf{U}|^{(n_{0}+n_{1}-m-1)/2} |\mathbf{F}|^{(n_{1}-m-1)/2} |\mathbf{U}|^{(n_{2}-m-1)/2}
  (d\mathbf{W})\wedge(d\mathbf{F})\wedge \left(d\mathbf{U}\right),
\end{equation}
where $\mathbf{W} > 0$, $\mathbf{F} > 0$, and $\mathbf{U} > 0$. This density will be termed
\emph{trimatricvariate Wishart-beta type II-beta type I distribution}.
\end{theorem}
\begin{proof}
The joint density function of $\mathbf{X}_{0},\mathbf{X}_{1},\mathbf{X}_{2}$ is
$$
   h\left\{\tr\left(\mathbf{X}'_{0}\mathbf{X}_{0} + \mathbf{X}'_{1}\mathbf{X}_{1} + \mathbf{X}'_{2}\mathbf{X}_{2}
  \right)\right\}(d\mathbf{X}_{0})\wedge (d\mathbf{X}_{1})\wedge(d\mathbf{X}_{2}).
$$
Define $\mathbf{W}_{0} = \mathbf{X}'_{0}\mathbf{X}_{0}$, hence
$$
  (d\mathbf{X}_{0}) =   2^{-m}|\mathbf{W}_{0}|^{(n_{0}-m-1)/2}(d\mathbf{W}_{0})(\mathbf{H}_{1}d\mathbf{H}_{1}),
$$
where $\mathbf{H}_{1} \in \mathcal{V}_{m,n_{0}}$. Now,  use
$$
  \int_{\mathbf{H}_{1} \in \mathcal{V}_{m,n_{0}}}(\mathbf{H}_{1}d\mathbf{H}_{1}) =
  \frac{2^{k}\pi^{n_{0}m/2}}{\Gamma_{m}[n_{0}/2]},
$$
the joint density function of $\mathbf{W}_{0},\mathbf{X}_{1},\mathbf{X}_{2}$ is
$$
  \frac{\pi^{n_{0}m/2}}{\Gamma_{m}[n_{0}/2]} |\mathbf{W}_{0}|^{(n_{0}-m-1)/2} h\left\{\tr\left(\mathbf{W}_{0}
   + \mathbf{X}'_{1}\mathbf{X}_{1}+ \mathbf{X}'_{2}\mathbf{X}_{2} \right)\right\}(d\mathbf{W}_{0})\wedge
   (d\mathbf{X}_{1})\wedge(d\mathbf{X}_{2}).
$$
Make the change of variables $\mathbf{T} = \mathbf{X}_{1}\mathbf{W}_{0}^{-1/2}$, $\mathbf{R}=
\mathbf{X}_{2}(\mathbf{W}_{0} + \mathbf{X}_{2}'\mathbf{X}_{2})^{-1/2}$ and $\mathbf{W} =
\mathbf{W}_{0}+\mathbf{X}'_{2}\mathbf{X}_{2}$. Therefore
$$
  (d\mathbf{W}_{0})\wedge (d\mathbf{X}_{1})\wedge(d\mathbf{X}_{2}) = |\mathbf{W}_{0}|^{n_{1}/2}
  |\mathbf{W}_{0}+ \mathbf{X}_{2}'\mathbf{X}_{2}|^{n_{2}/2}(d\mathbf{W})\wedge
  (d\mathbf{T})\wedge(d\mathbf{U}).
$$
Note that
$$
  |\mathbf{W}_{0}+ \mathbf{X}_{2}'\mathbf{X}_{2}| = |\mathbf{W}| \mbox{ and } |\mathbf{W}_{0}| =
  |\mathbf{W}||\mathbf{I}_{m}- \mathbf{R}'\mathbf{R}|.
$$
Also observe that
$$
  \tr \mathbf{W}_{0} = \tr \mathbf{W}(\mathbf{I}_{m}- \mathbf{R}'\mathbf{R}),
$$
$\tr \mathbf{X}'_{1}\mathbf{X}_{1} = \tr \mathbf{W}^{1/2}(\mathbf{I}_{m}-
\mathbf{R}'\mathbf{R})\mathbf{W}^{1/2} \mathbf{T}'\mathbf{T}$ and $\tr \mathbf{X}'_{2}\mathbf{X}_{2} =
\tr \mathbf{W}\mathbf{R}'\mathbf{R}$. Then, the first required result is obtained. Finally, if
$\mathbf{F} = \mathbf{T}'\mathbf{T}$ and $\mathbf{U} = \mathbf{R}'\mathbf{R}$, then the exterior product
$(d\mathbf{W})\wedge(d\mathbf{T})\wedge \left(d\mathbf{R}\right)$ is
$$
  = 2^{-2m} |\mathbf{F}|^{(n_{1}-m-1)/2} |\mathbf{U}|^{(n_{2}-m-1)/2} (\mathbf{H}'_{1}d\mathbf{H}_{1})
  (\mathbf{P}'_{1}d\mathbf{P}_{1})(d\mathbf{W})\wedge(d\mathbf{F})\wedge \left(d\mathbf{U}\right).
$$
And integrating $(\mathbf{H}'_{1}d\mathbf{H}_{1})$ and $(\mathbf{P}'_{1}d\mathbf{P}_{1})$ over
$\mathbf{H}_{1} \in \mathcal{V}_{m,n_{1}}$ and $\mathbf{P}_{1} \in \mathcal{V}_{m,n_{2}}$, respectively,
the second distribution is obtained.
\end{proof}

When the matrices in a multimatricvariate distribution are non singular, we can derive  the corresponding
inverse distributions. For example:

Suppose that $\mathbf{V}_{1}, \dots,\mathbf{V}_{k_{1}},\mathbf{V}_{k_{1}+1}, \dots,\mathbf{V}_{k_{2}}$
follow a multimatricvariate generalised Wishart distribution. Then their joint density  is
$$
    \pi^{Nm/2} \left(\prod_{i=1}^{k_{1}} \frac{|\mathbf{V}_{i}|^{(n_{i}-m-1)/2}}{\Gamma_{m}[n_{i}/2]
    |\mathbf{\Sigma}_{i}|^{n_{i}/2}}\right ) \left(\prod_{j=k_{1}+1}^{k2} \frac{|\mathbf{V}_{j}|^{(n_{j}-m-1)/2}}{\Gamma_{m}[n_{j}/2]
    |\mathbf{\Sigma}_{j}|^{n_{j}/2}}\right )\hspace{2cm}
$$
$$\hspace{2cm}
   \times h\left(\tr\left(\mathbf{\Sigma}_{1}^{-1}\mathbf{V}_{1}+ \cdots
    + \mathbf{\Sigma}_{k}^{-1}\mathbf{V}_{k_{1}}+\mathbf{\Sigma}_{k_{1}+1}^{-1}\mathbf{V}_{k_{1}+1}+ \cdots
    + \mathbf{\Sigma}_{k}^{-1}\mathbf{V}_{k_{2}}\right)\right)
$$
$$
   \hspace{9cm} \times   \bigwedge_{i=1}^{k_{1}}(d\mathbf{V}_{i})\bigwedge_{j=k_{1}+1}^{k_{2}}(d\mathbf{V}_{j}).
$$
Consider the change of variable $\mathbf{W}_{j} = \mathbf{V}_{j}^{-1}$, $j = k_{1}+1, \dots,k_{2}$, then
by \citet[Theorem 2.1.8, p. 65]{mh:05},
$$
  \bigwedge_{j=k_{1}+1}^{k_{2}}(d\mathbf{V}_{j}) = \prod_{j=k_{1}+1}^{k_{2}}|\mathbf{W}_{j}|^{-(m+1)/2}
  \bigwedge_{j=k_{1}+1}^{k_{2}}(d\mathbf{W}_{j}).
$$
Hence the joint density $dF_{\mathbf{V}_{1}, \dots,\mathbf{V}_{k_{1}},\mathbf{W}_{k_{1}+1},
\dots,\mathbf{W}_{k_{2}}}(\mathbf{V}_{1}, \dots,\mathbf{V}_{k_{1}},\mathbf{W}_{k_{1}+1},
\dots,\mathbf{W}_{k_{2}})$ is
\begin{equation}\label{mmWiW}
     \pi^{Nm/2} \left(\prod_{i=1}^{k_{1}} \frac{|\mathbf{V}_{i}|^{(n_{i}-m-1)/2}}{\Gamma_{m}[n_{i}/2]
    |\mathbf{\Sigma}_{i}|^{n_{i}/2}}\right ) \left(\prod_{j=k_{1}+1}^{k2} \frac{|\mathbf{W}_{j}|^{-(n_{j}+m+1)/2}}{\Gamma_{m}[n_{j}/2]
    |\mathbf{\Sigma}_{j}|^{n_{j}/2}}\right )\hspace{2cm}
$$
$$\hspace{2cm}
   \times h\left(\tr\left(\mathbf{\Sigma}_{1}^{-1}\mathbf{V}_{1}+ \cdots
    + \mathbf{\Sigma}_{k}^{-1}\mathbf{V}_{k_{1}}+\mathbf{\Sigma}_{k_{1}+1}^{-1}\mathbf{W}^{-1}_{k_{1}+1}+ \cdots
    + \mathbf{\Sigma}_{k}^{-1}\mathbf{W}^{-1}_{k_{2}}\right)\right)
$$
$$
   \hspace{9cm} \times   \bigwedge_{i=1}^{k_{1}}(d\mathbf{V}_{i})\bigwedge_{j=k_{1}+1}^{k_{2}}(d\mathbf{W}_{j}).
\end{equation}
This function will be termed \emph{multimatricvariate generalised Wishart-inverted Wishart distribution}.

In a similar way, assume that $\mathbf{F}_{1}, \dots,\mathbf{F}_{k_{1}},\mathbf{F}_{k_{1}+1},
\dots,\mathbf{F}_{k_{2}}$ have the density (\ref{MMVF}), then we obtain the so called
\emph{multimatricvariate generalised beta type I-inverted beta type I distribution}. Both results are
summarized next.

\begin{theorem}
Assume that $\mathbf{V}_{1}, \dots,\mathbf{V}_{k_{1}},\mathbf{V}_{k_{1}+1}, \dots,\mathbf{V}_{k_{2}}$
follow a multimatricvariate generalised Wishart distribution. And define $\mathbf{W}_{j} =
\mathbf{V}_{j}^{-1}$, $j = k_{1}+1, \dots,k_{2}$, then the joint density function of $\mathbf{V}_{1},
\dots, \mathbf{V}_{k_{1}},\mathbf{W}_{k_{1}+1}, \dots,\mathbf{W}_{k_{2}}$ is termed
\emph{multimatricvariate generalised Wishart-inverted Wishart distribution} and is given by
(\ref{mmWiW}). Similarly, if
$$
  \mathbf{F}_{1}, \dots,\mathbf{F}_{k_{1}},\mathbf{F}_{k_{1}+1},\dots,\mathbf{F}_{k_{2}}
$$
have  a \emph{multimatricvariate generalised beta type I distribution} and we
define $\mathbf{E}_{j} = \mathbf{F}_{j}^{-1}$ with $j = k_{1}+1, \dots,k_{2}$, then the joint density
function
$$
  dF_{\mathbf{F}_{1}, \dots,\mathbf{F}_{k_{1}},\mathbf{E}_{k_{1}+1},
\dots,\mathbf{E}_{k_{2}}}(\mathbf{F}_{1}, \dots,\mathbf{F}_{k_{1}},\mathbf{E}_{k_{1}+1},
\dots,\mathbf{E}_{k_{2}})
$$
is given by
\begin{equation}\label{MMFiF}
     \frac{\Gamma_{m}[n^{*}/2]}{\displaystyle\prod_{i=0}^{k_{2}}\Gamma_{m}[n_{i}/2]}
  \frac{\displaystyle\prod_{i=1}^{k_{1}}|\mathbf{F}_{i}|^{(n_{i}-m-1)/2}\prod_{j=k_{1}+1}^{k_{2}}|\mathbf{E}_{j}|^{-(n_{j}+m+1)/2}}{
  \left|\mathbf{I}_{m}+\displaystyle\sum_{i=1}^{k_{1}}\mathbf{F}_{i}+\sum_{j=k_{1}+1}^{k_{2}}\mathbf{E}^{-1}_{j}\right|^{n^{*}/2}}
  \bigwedge_{i=1}^{k_{1}}\left(d\mathbf{F}_{i}\right) \bigwedge_{i=k_{1}+1}^{k_{2}}\left(d\mathbf{E}_{j}\right),
\end{equation}
where $n^{*} = n_{0}+n_{1}+\cdots+n_{k_{1}}+n_{k_{1}+1}+\cdots+n_{k_{2}}$. This function will be termed
\emph{multimatricvariate generalised beta type I-inverted beta type I distribution}.
\end{theorem}

\section{Example}\label{sec:5}

Consider a sample $\mathbf{F}_{1}, \dots,\mathbf{F}_{k}$, of $2 \times 2$ random matrices taken from  a
matricvariate beta type II distribution with parameters $n_{0}$ and $n_{1}$.  By Remark \ref{rem1} we
extend the parametric space from $(n_{0},n_{1})$ to $(a_{0},a)$, then the corresponding likelihood
function $L(a_{0}, a; \mathbf{F}_{1}, \cdots, \mathbf{F}_{k})$ is
$$
   \left\{
  \begin{array}{ll}\displaystyle
    \prod_{j=1}^{k} f_{\mathbf{F}_{j}}(\mathbf{F}_{i}; a_{0},a),
     & \hbox{if the $\mathbf{F}_{i's}$ are independent;} \\
    f_{\mathbf{F}_{1}, \cdots, \mathbf{F}_{k}}(\mathbf{F}_{1}, \cdots, \mathbf{F}_{k};a_{0},a),
    & \hbox{if the $\mathbf{F}_{i's}$ are dependent.}
  \end{array}
\right.
$$
Here $f_{\mathbf{F}_{1}, \cdots, \mathbf{F}_{k}}(\mathbf{F}_{1}, \cdots, \mathbf{F}_{k};a_{0}, a)$ follow
from (\ref{MMVF}) by taking $a_{1}=\cdots=a_{k} = a$. And $f_{\mathbf{F}_{i}}(\mathbf{F}_{i};a_{0},a)$
comes from (\ref{MMVF}) by setting $k=1$ and denoting $a_{i} = a$ for all $i= 1,\dots,k$. Explicitly, the
likelihood function is:
$$
    \frac{\left(\Gamma_{m}[a_{0}+a]\right)^{k}}{\left(\Gamma_{m}[a_{0}]\right)^{k}\left(\Gamma_{m}[a]\right)^{k}}
  \prod_{j=1}^{k}|\mathbf{F}_{j}|^{a-(m+1)/2}\prod_{j=1}^{k}\left|\mathbf{I}_{m}+\mathbf{F}_{j}\right|^{-(a_{0}+a)},
$$
in the independent case, and
$$
   = \frac{\Gamma_{m}[a_{0}+ka]}{\Gamma_{m}[a_{0}](\Gamma_{m}[a])^{k}}
  \prod_{j=1}^{k}|\mathbf{F}_{j}|^{a-(m+1)/2}
  \left|\mathbf{I}_{m}+\displaystyle\sum_{j=1}^{k}\mathbf{F}_{j}\right|^{-(a_{0}+ka)},
$$
in the dependent case. Recall that $a^{*} = a_{0}+ka$.

Inference with real data in  matrix variate distribution theory  is rarely published. Thus, in this
section we  show that performing inference with the derived densities  is a possible task. There are a
number of data that we can use. In particular the statistical theory of shape provides interesting
situations to be applied. Most of the applications involve time independent configurations of landmarks
in two dimensions. In this example we consider the location in angstroms of a part of a 3D DNA molecule
with 22 phosphorous atoms. The positions of the atoms are reported in 30 different times. It is clear
that the configurations are highly correlated, then a dependent model rules the motion of the molecule.
Statistical theory of shape can describe for example the  shape evolution or the shape variability in
mean and covariance, or the configuration space, among many others, see \citet{DM98} and the references
therein.

The original sample consists of $k = 30$ matrices of order $22\times 3$,  which turns into a sample of
$3\times 3$  beta type II matrices, $\mathbf{F}_{1}, \dots, \mathbf{F}_{30}$. We want to estimate the
parameters $a_{0}$ and $a$ under the realistic dependent model. However, we also give the estimation of
the independent case for a reference of the involved scales of the parameters and the numerical
performance.

As we expect,  the optimization algorithm  depends on feasible seeds.  Given that we do not have
suggestions for an appropriate domain of $a_{0}$ and $a$, we use the corresponding estimates in the
univariate beta type II distribution.

Using standard packages for local  optimization, as Optimx of R, and Price's method for global
optimization, we find the following maximum likelihood (ML) estimation of the $3\times 3$
multimatricvariate beta type II distribution under dependence:
$$
a_{0}=0.04724072\quad\mbox{and}\quad a=0.1595563.
$$
With that model we can perform, for example, simulations of the movement of the associated matricvariate
beta type II random matrix of the molecule. The inverse problem also can be considered in order to
describe the shape variability of the original $22\times 3$ configuration  throughout larger scales of
time by using a technique provided in theorem 3.8 of \cite{dgcl:17}. In this case we use the fact that
$\mathbf{F} = \mathbf{X}'\mathbf{X}$ and we can return to $\mathbf{X}$ from $\mathbf{F}$.

Only for a comparison of the scales in the ML estimates, the local and global optimization routines in
the independence model converge to:
$$
a_{0}=2.068213\quad\mbox{and}\quad a=4101.160.
$$

Recall that both models cannot be compared because they correspond to different situations, in this case,
the intrinsic dependence in the objects is out of any discussion.

Finally, note that the multimatricvariate beta type II distribution can be seen as  a generalized Wishart
distribution.  Let $\mathbf{X}$ be a matricvariate T distribution, then the matrix $\mathbf{F} =
\mathbf{X}'\mathbf{X}$ follows a generalized  Wishart distribution which coincides with the
multimatricvariate beta type II distribution, \citet{di:67}.

\section{Conclusions}
\begin{enumerate}
  \item Multiple partitions of a matrix variate elliptical distribution generate an uncountable list of
     dependent models with a priori defined marginal distributions.
  \item The extended parametric spaces in  multimatricvariate distributions allow new perspectives for
     dependent models. They also propose new challenges for geometrical and statistical interpretation.
  \item The multimatricvariate distributions can be used for likelihood functions of dependent random samples.
    This novelty seems to be out of the multimatrix literature, which is based on certain independent scenarios.
\end{enumerate}



\end{document}